\documentclass[11pt]{article}

\usepackage{epic,eepic,epsfig,amssymb,amsmath,amsthm,graphics}

\usepackage{multimedia}

\usepackage{pdftricks}

\begin{psinputs}
\usepackage{pstricks,pst-text}
\end{psinputs}

\usepackage{amsmath,amssymb,amsthm,amsfonts,mathrsfs}
\def\R{\mathbb R}
\def\G{\mathbb G}
\def\N{\mathbb N}
\def\H{\mathbb H}
\def\Z{\mathbb Z}

\def\MCP{{\rm MCP}}
\def\TIL{{\rm TIL}}
\def\cut{{\rm cut}}
\def\vol{{\rm vol}}
\def\Ric{{\rm Ric}}

\def\ad{{\rm ad}}

\numberwithin{equation}{section}

\newtheorem{theorem}{Theorem}

\newtheorem{lemma}[theorem]{Lemma}
\newtheorem{proposition}[theorem]{Proposition}

\newtheorem{definition}[theorem]{Definition\rm}

\makeindex

\begin{document}

\title{Ricci curvatures in Carnot groups}
%\footnotetext{The author has been supported by the program GCM''.}

%%%%%%%%%%%%%%%%%%%%%%%%%%%%%%%%%%%%%%%%%%%%%%%%%%%%%%%%%%%%%%%%%%%

\author{L.~Rifford\thanks{Universit\'e de Nice-Sophia
    Antipolis, Labo.\ J.-A.\ Dieudonn\'e, UMR CNRS 6621, Parc
    Valrose, 06108 Nice Cedex 02, France ({\tt
      Ludovic.Rifford@math.cnrs.fr})}}

%\date{Preliminary version of April 15, 2010}

%%%%%%%%%%%%%%%%%%%%%%%%%%%%%%%%%%%%%%%%%%%%%%%%%%%%%%%%%%%%%%%%%%

%\makeindex

\maketitle

\begin{abstract}
We study metric contraction properties for metric spaces associated with left-invariant sub-Riemannian metrics on Carnot groups. We show that ideal sub-Riemannian structures on Carnot groups satisfy such properties and give a lower bound of possible curvature exponents in terms of the datas.
\end{abstract}

%\tableofcontents

\section{Introduction}\label{SECintroduction}
\label{Intro}
Several notions of Ricci curvature (bounded from below) on measured metric spaces have been defined recently \cite{lv07,ohta07,sturm06a,sturm06b}. The Ohta measure contraction property \cite{ohta07} reflects the way the volume of balls is distorted along geodesics. In \cite{juillet09}, Juillet proved that Heisenberg groups equipped with their canonical sub-Riemannian metric and the Haar measure satisfy some measure contraction properties. The aim of the present paper is to extend Juillet's results to more general Carnot groups. In order to introduce Ohta's measure contraction property, we first study the Riemannian case.  We refer the reader to the textbook \cite{ghl04} for further details in Riemannian geometry.\\

Let $(M,g)$ be a  complete smooth Riemannian manifold of dimension $n$, we denote by $d_g$ the geodesic distance on $M$, by $d\vol_g$ the canonical measure of $(M,g)$ and by $\exp_x:T_xM \rightarrow M$ the exponential map from a point $x\in M$ . Given $x\in M$, the cut locus of $x$, denoted by $\cut(x)$, is the smallest closed set in $M$ such that the pointed distance $y \in M\mapsto d_g(x,y)$ is smooth on its complement in $M$. That set can also be viewed as the closure of the set of points $y\in M$ such that at least two distinct minimizing geodesics join $x$ to $y$. For every $y\in M\setminus \cut(x)$, there is a unique minimizing geodesic joining $x$ to $y$; we denote by $v_y\in T_xM$ the tangent vector such that $\exp_x(v_y)=y$ and by $\gamma_{x,v_y} :[0,1] \rightarrow M$ the geodesic starting at $x$ with velocity $v_y$. The tangent injectivity locus at $x$ is the subset of $T_xM$ defined as 
$$
\TIL(x) := \Bigl\{ v_y \, \vert \, y \in M \setminus \cut(x)\Bigr\};
$$
it is an open subset of $T_xM$ which is star-shaped with respect to the origin and has a locally Lipschitz boundary \cite{cr10,it01,ln05}. Moreover, the mapping 
$$
\exp_x \, : \, \TIL(x) \, \longrightarrow \, M\setminus \cut(x)
$$
is a smooth diffeomorphism. The set $\cut(x)$ is the image of $\partial \TIL(x)$ through $\exp_x$, hence it has measure zero. Let $x\in M$ and $A\subset M$ a measurable set with $0<\vol_g(A) < \infty$ be fixed. For every $s\in [0,1]$, we call $s$-interpolation of $A$ from $x$, the subset of $M$ defined by
$$
A_s := \Bigl\{ \gamma_{x,v}(s) \, \vert \, v \in \TIL(x), \, \exp_x(v) \in A\Bigr\}.
$$
\begin{pdfpic}
\begin{pspicture}(10,3)
%\psgrid
\psdots[dotsize=5pt](2,1)
\rput(2,0.75){$x$}
\psccurve[linewidth=1pt](3.5,0.8)(4,1.05)(4.2,1.8)(3.5,1.7)
\psccurve[linewidth=1pt,fillstyle=solid,fillcolor=lightgray](7,1)(7.5,1.3)(8,3)(7,2.9)
\pscurve[linewidth=1pt,linestyle=dashed](2,1)(4.2,2.08)(7.5,3.17)
\pscurve[linewidth=1pt,linestyle=dashed](2,1)(5,0.65)(7.25,1)
\rput(7.75,1.25){$A$}
\rput(4.3,1){$A_s$}
\end{pspicture}
\end{pdfpic}
\begin{center}
Figure 1
\end{center}

Note that $A_0=\{x\}$ while $A_1=A\setminus \cut(x)$. Denote by $U_xM$ the unit sphere in $(T_xM,g_x)$, and for every $s\in [0,1]$ by $D_s$ the subset of $T_xM$ corresponding to $A_s$ through the exponential mapping in polar coordinates, that is 
$$
D_s:= \Bigl\{ (t,u)  \in [0,+\infty) \times U_xM \, \vert  \,  \exp_x(tu) \in A_s \Bigr\}.
$$
Let $u\in U_xM$ and $\left( e_1, \ldots, e_n\right)$ be an orthonormal basis of $T_xM$ with $e_1=u$. For every $t\geq 0$ let $\gamma_{x,u}(t)= \exp_x(tu)$, and let $\left( e_1(t),\ldots, e_n(t)\right)$ be an orthonormal basis of $T_{\gamma_{x,u}(t)}M$ obtained by parallel transport of $\left( e_1, \ldots, e_n\right)$ along $\gamma_{x,u}$. Let further, for $t\geq 0$,
\begin{eqnarray*}
R_{ij}(t,u) = \Bigl\langle \mbox{Riem}_{\gamma_{x,u}(t)} \bigl(\dot{\gamma}_{x,u}(t), e_i(t)\bigr) \dot{\gamma}_{x,u}(t), e_j(t)\Bigr\rangle_{\gamma_{x,u}(t)},\qquad 1\leq i,j\leq n,
\end{eqnarray*}
where $\mbox{Riem}$ stands for the Riemann curvature tensor. Note that the $n\times n$ matrix $R(t,u) = \left( R_{ij}(t,u)\right) $ is symmetric and satisfies $R_{i1}(t,u)=R_{1i}(t,u)=0$ for any $i$. We define $J(t,u)$, implicitly depending on $x$ and $\left(e_1,\ldots, e_n\right)$, as the matrix-valued solution of
\begin{eqnarray*}
\begin{cases}
\ddot{J}(t,u) + R(t,u) J(t,u) = 0, \\[1mm]
J(0,u)=0_n,\quad \dot{J}(0,u) = I_n.
\end{cases}
\end{eqnarray*}
The first line and row of $J$ always satisfy $J_{i1}(t,u)=J_{1i}(t,u)= t\delta_{i1}$ for any $t\geq 0$ and any $i$. Then the $(n-1)\times (n-1)$ block in the lower right of $J$, that we denote by $\hat{J}$, is solution to 
\begin{eqnarray}\label{jacobi2}
\begin{cases}
\ddot{\hat{J}}_i(t,u) + \hat{R}(t,u) \hat{J}_i(t,u) = 0, \\[1mm]
\hat{J}(0,u)=0_{n-1},\quad \dot{\hat{J}}(0,u) = I_{n-1},
\end{cases}
\end{eqnarray}
where $\hat{R}(t,u)$ denotes the $(n-1)\times (n-1)$ block in the lower right of $R(t,u)$. Note that $\hat{J}(t,u)$ may depend on the orthonormal basis $\left( e_1, \ldots, e_n\right)$ of $T_xM$, however its determinant does not. Then we set 
\begin{eqnarray}\label{DEFD}
D(t,u) := \det \bigl(\hat{J}(t,u)\bigr) \qquad \forall t \geq 0, \, \forall u \in U_xM,
\end{eqnarray}
which depends implicitely on $x$. Since $t^{-(n-1)} D(t,u)$ corresponds to the Jacobian determinant of the mapping $(t,u) \mapsto \exp_x(tu)$, we get by change of variable
\begin{eqnarray*}
\vol_g (A) := \int_{A} 1 \, d\vol_g(z)  = \int_{A_1} 1 \, d\vol_g(z)  =  \int_{D_1} D(t,u)  \, dt \, du,
\end{eqnarray*}
and for every $s\in [0,1]$,
\begin{eqnarray}\label{As}
\vol_g(A_s)  & = & \int_{A_s} 1 \, d\vol_g(z) = \int_{D_s} D(t,u) \, dt \, du = \int_{D_1} s D(st,u)  \, dt \, du.
\end{eqnarray}
In the above change of variable, we used that for every $u\in U_xM$, the determinant of the Jacobi matrix $J(t,u)$ (or equivalently of $\hat{J}(t,u)$) is positive as long as the tangent vector $tu$ (with $t>0$) belongs to $\TIL(x)$. For every $u \in U_xM$, we denote by $t_{\cut}(u)$ the first time $t\geq 0$ such that $tu\notin \TIL(x)$ (if $tu$ always belongs to $\TIL(x)$ we set $t_{\cut}(u)=+\infty$). Given $u\in U_xM$, remembering (\ref{jacobi2}) and that $R$ (and a fortiori $\hat{R}$) is symmetric, we check easily that  the $(n-1)\times (n-1)$ matrix 
$$
U(t,u):= \dot{\hat{J}}(t,u) \hat{J}(t,u)^{-1} \qquad \forall t \in \bigl(0,t_{\cut}(u)\bigr)
$$
satisfies the Ricatti equation
\begin{eqnarray}\label{ricatti}
\dot{U} (t,u) + U(t,u)^2 + \hat{R}(t,u) = 0 \qquad \forall t \in \bigl(0,t_{\cut}(u)\bigr),
\end{eqnarray}
and is symmetric. Since the derivative of a determinant is a trace, we have 
$$
 \frac{\dot{D}(t,u)}{D(t,u)} = \mbox{tr} \bigl( U(t,u)\bigr)  \qquad \forall t \in \bigl(0,t_{\cut}(u)\bigr).
$$
Moreover, the Cauchy-Schwarz inequality yields 
$$
\left( \mbox{tr} \bigl(U(t,u)\bigr) \right)^2 \leq \mbox{tr} \left( U(t,u)^2 \right) (n-1).
$$  
Then taking the trace of (\ref{ricatti}), we get 
$$
\mbox{tr} \left( \dot{U} (t,u)\right) + \frac{1}{n-1} \left[ \mbox{tr} \left( U(t,u)\right) \right]^2 + \mbox{tr} \bigl( \hat{R}(t,u)\bigr) \leq 0 \qquad \forall t \in \bigl(0,t_{\cut}(u)\bigr).
$$
Recall that we have by definition of the Ricci curvature  as a quadratic form
$$
\Ric_g \left(\dot{\gamma}_{x,u} (t)\right) =  \mbox{tr} \left( \hat{R}(t,u)\right)  \qquad \forall t \in \bigl(0,t_{\cut}(u)\bigr).
$$
Therefore if a Riemannian manifold satisfies $\Ric_g \geq K$, we have 
$$
\frac{d}{dt} \left( \frac{\dot{D}(t,u)}{D(t,u)}\right) + \frac{1}{n-1} \, \left( \frac{\dot{D}(t,u)}{D(t,u)}\right)^2 +K \leq 0  \qquad \forall t \in \bigl(0,t_{\cut}(u)\bigr).
$$
By classical comparison theorems, we infer that
\begin{eqnarray}\label{comparison}
 \frac{\dot{D}(t,u)}{D(t,u)} \leq \sqrt{n-1}\, \frac{\dot{s}_K\left(t/\sqrt{n-1}\right)}{s_K\left(t/\sqrt{n-1}\right)}  \qquad \forall t \in \bigl(0,t_{\cut}(u)\bigr),
\end{eqnarray}
where the function $s_K: [0,+\infty) \rightarrow [0,+\infty)$ ($s_K: [0,\pi/\sqrt{K})\rightarrow [0,+\infty)$ if $K>0$) is defined by
$$
s_K(t) := \left\{ \begin{array}{lll}
\frac{\sin(\sqrt{K}t)}{\sqrt{K}} & \mbox{ if } K>0\\
t & \mbox{ if } K=0 \\
\frac{\sinh(\sqrt{-K}t)}{\sqrt{-K}} & \mbox{ if } K<0.
\end{array}
\right.
$$
Then, for every $s\in (0,1)$ integrating the inequality on $[st,t]$ yields (note that by Bonnet-Myers's Theorem, $K>0$ implies $t_{\cut}(u) < \pi\sqrt{n-1/K}$ for all $u$)
$$
\frac{D(t,u)}{D(st,u)} \leq \left[  \frac{s_K\left(t/\sqrt{n-1}\right)}{s_K\left(st/\sqrt{n-1}\right)} \right]^{n-1}  \qquad \forall t \in \bigl(0,t_{\cut}(u)\bigr).
$$
Then remembering (\ref{As}), we get for every $s\in (0,1)$,
\begin{eqnarray*}
\vol_g \left( A_s\right) & = &  \int_{D_1} sD(st,u) \, dt \, du\\
& \geq & s \int_{D_1} \left[  \frac{s_K\left(st/\sqrt{n-1}\right)}{s_K\left(t/\sqrt{n-1}\right)} \right]^{n-1}  D(t,u) \, dt \, du\\
& = &  \int_{A_1} s \left[  \frac{s_K\left(sd_g(x,z)/\sqrt{n-1}\right)}{s_K\left(d_g(x,z)/\sqrt{n-1}\right)} \right]^{n-1}  \, d\vol_g(z) \\
& = &  \int_{A} s \left[  \frac{s_K\left(sd_g(x,z)/\sqrt{n-1}\right)}{s_K\left(d_g(x,z)/\sqrt{n-1}\right)} \right]^{n-1}  \, d\vol_g(z).
\end{eqnarray*}
Note that the above inequality holds for $s=0, 1$. Then summarizing we have the following result.

\begin{proposition}\label{PROPriem}
Let $(M,g)$ be a complete smooth Riemannian manifold satisfying $\Ric_g \geq K$. Then for every $x\in M$ and every measurable set $A \subset M$ with $0<\vol_g(A) < \infty$, we have 
\begin{eqnarray}\label{PROP1eq}
\vol_g \left( A_s\right) \geq  \int_{A} s \left[  \frac{s_K\left(sd_g(x,z)/\sqrt{n-1}\right)}{s_K\left(d_g(x,z)/\sqrt{n-1}\right)} \right]^{n-1}  \, d\vol_g(z) \qquad \forall s \in [0,1].
\end{eqnarray}
\end{proposition}

According to the above result, Ohta introduced in \cite{ohta07} the notion of measure contraction property of general measured metric spaces that we proceed to define. For sake of simplicity we restrict our attention to measured metric spaces with negligeable cut loci.

\begin{definition}
Let $(X,d,\mu)$ be a measured metric space. We say that it is a geodesic space with negligeable cut loci if for every $x\in X$, there is a measurable set $\mathcal{C}(x) \subset X$ with  
$$
\mu \left( \mathcal{C}(x)\right) =0,
$$
and a measurable map $\mathcal{E}_{x} \, : \, \left( X \setminus \mathcal{C}(x)\right) \times [0,1] \longrightarrow X$ such that for every $y \in  X \setminus \mathcal{C}(x)$ the curve 
$$
s \in [0,1] \longmapsto \mathcal{E}_x (y,s)
$$
is the unique geodesic from $x$ to $y$. 
\end{definition}

The following definition is equivalent to Ohta's definition \cite[Definition 2.1]{ohta07} in the case of geodesic measured metric spaces with negligeable cut loci.

\begin{definition}
Let $(X,d,\mu)$ be a measured metric space which is geodesic with negligeable cut loci and $K\in \R, N> 1$ be fixed. We say that $(X,d,\mu)$ satisfies $\MCP(K,N)$ if for every $x\in X$ and every measurable set $A\subset X$ (provided that $A\subset B_d(x,\pi\sqrt{N-1/K})$ if $K>0$) with $0< \mu(A)<\infty$,
$$
\mu \left(A_s\right)  \geq  \int_{A} s \left[  \frac{s_K\left(sd(x,z)/\sqrt{N-1}\right)}{s_K\left(d(x,z)/\sqrt{N-1}\right)} \right]^{N-1}  \, d\mu(z).
$$
where $A_s$ is the $s$-interpolation of $A$ from $x$ defined by
$$
A_s := \Bigl\{ \mathcal{E}_x(y,s) \, \vert \, y \in A \setminus \mathcal{C}(x) \Bigr\} \qquad \forall s \in [0,1].
$$
In particular, $(X,d,\mu)$ satisfies $\MCP(0,N)$ if for every $x\in X$ and every measurable set $A$ with $0< \mu(A)<\infty$,
$$
\mu \left(A_s\right) \geq s^N \mu (A) \qquad \forall s \in [0,1],
$$
\end{definition}

Of course, Euclidean spaces, that is $\R^n$ equipped with a constant Riemannian metric, satisfy $\MCP(0,n)$. Carnot groups are to sub-Riemannian geometry as Euclidean spaces are to Riemannian geometry. They are the metric tangent cones for this geometry. Elaborating on an idea of Gromov \cite{gromov99}, Mitchell \cite{mitchell85} proved that any sub-Riemannian structure does admit at generic points metric tangent cones which are Carnot groups equipped with  left-invariant metrics. This property makes them prime canditates to satisfy $\MCP(0,N)$. In \cite{juillet09}, Juillet  proved that the Heisenberg group $\H_n$ equipped with its canonical sub-Riemannian metric and the Haar measure satisfies $\MCP(0,N)$ with $N=2n+3$. Heisenberg groups are the most simple examples of sub-Riemannian structures. The aim of the present paper is to show the validity of Juillet's result for more general Carnot groups equipped with left-invariant sub-Riemannian structures and Haar measures. Let us now present briefly our results. We refer the reader to Section \ref{SECprelim} for reminders in sub-Riemannian geometry and Carnot groups.

\begin{theorem}\label{THM1}
Let $\G$ be a Carnot group whose first layer is equipped with a left-invariant metric, assume that it is ideal. Then  there is $N>0$ such that the metric space $(\G,d_{SR})$ with Haar measure satisfies $\MCP(0,N)$. 
\end{theorem}

We call curvature exponent of a Carnot group $\G$ whose first layer is equipped with a left-invariant metric the least $N\geq 1$ such that $\MCP(0,N)$ is satisfied. The curvature exponent is $\infty$ if $\MCP(0,N)$ is never satisfied for $N>1$. Note that if $N$ is finite, then $\G$ (equipped with its sub-Riemannian structure) does satisfy $\MCP(0,N)$. 

\begin{theorem}\label{THM2}
Let $\G$ be a Carnot group (equipped with a sub-Riemannian structure and the Haar measure), assume that  it is a geodesic space with negligeable cut loci. Then its curvature exponent $N$ satisfies
$$
N \geq D + n -  m,
$$
where $n$ is the topological dimension of $\G$, $D$ its homogeneous dimension and $m$ is the dimension of the first layer.
\end{theorem}

The paper is organized as follows. In Section \ref{SECprelim}, we recall some facts in sub-Riemannian geometry and Carnot groups theory. In particular, we recall important results regarding ideal sub-Riemannian structures. The proofs of Theorems \ref{THM1} and \ref{THM2} are given in Section \ref{SECproof}. The last section contains comments. 

\section{Preliminaries}\label{SECprelim}

\subsection{Sub-Riemannian structures}

Let us first recall basic facts in sub-Riemannian geometry, we refer the reader to \cite{abb12,fr10,montgomery02,riffordbook} for further details. Let $M$ be a smooth connected manifold without boundary of dimension $n\geq 3$, a sub-Riemannian structure on $M$ is given by a pair $(\Delta,g)$ where $\Delta$ is a totally nonholonomic distribution with constant rank $m\in [2,n]$ on $M$ and $g$ is a smooth Riemannian metric on $\Delta$. A path $\gamma : [0,1] \rightarrow M$ is called horizontal if it belongs to $W^{1,2}\left( [0,1] ; M\right)$ and satisfies 
$$
\dot{\gamma} (t) \in \Delta \left( \gamma(t)\right) \qquad \mbox{a.e. } t \in [0,1].
$$
From the Chow-Rashevsky Theorem, any points $x,y \in M$ can be joined by an horizontal path. For every $x\in M$ and any $v\in \Delta(x)$, we denote by $|v|_x^g$ the norm of $v$ with respect to the metric $g$. The length of an horizontal path $\gamma \in W^{1,2} \left( [0,1];M\right)$ is defined as
$$
\mbox{length}^g (\gamma) := \int_0^1 \bigl| \dot{\gamma}(t)\bigr|_{\gamma(t)}^g \, dt.
$$
For every $x, y \in M$, the sub-Riemannian distance between $x$ and $y$, denoted by $d_{SR}(x,y)$, is defined as the infimum of lengths of horizontal paths joining $x$ to $y$, that is 
$$
d_{SR}(x,y) := \inf \Bigl\{ \mbox{length}^g (\gamma) \, \vert \, \gamma \in W^{1,2} \bigl( [0,1];M\bigr), \, \gamma(0)=x, \, \gamma(1)=y \Bigr\}.
$$
The function $d_{SR}$ makes $(M,d_{SR})$ a metric space. The energy of an horizontal path $\gamma \in W^{1,2} \left( [0,1];M\right)$ is defined as 
$$
\mbox{energy}^g (\gamma) := \int_0^1 \left( \bigl| \dot{\gamma}(t)\bigr|_{\gamma(t)}^g \right)^2 \, dt.
$$
So the energy between $x$ and $y$ in $M$ is defined as
$$
e_{SR}(x,y) := \inf \Bigl\{ \mbox{energy}^g (\gamma) \, \vert \, \gamma \in W^{1,2} \bigl( [0,1];M\bigr), \, \gamma(0)=x, \, \gamma(1)=y \Bigr\}.
$$
The Cauchy-Schwarz inequality implies easily $e_{SR} = d_{SR}^2$ on $M\times M$. By the sub-Riemannian Hopf-Rinow Theorem, if $(M,d_{SR})$ is assumed to be complete, then for every $x,y\in M$ there exists at least one minimizing geodesic joining $x$ to $y$, that is an horizontal path $\gamma :[0,1] \rightarrow M$ with $\gamma(0)=x, \gamma(1)=y$ satisfying
$$
d_{SR}(x,y) =    \mbox{length}^g (\gamma) = \sqrt{\mbox{energy}^g (\gamma)}.
$$
We need now to introduce the notion of singular horizontal curves. For sake of simplicity we restrict our attention to minimizing geodesic curves. Let $x, y \in M$ and a minimizing geodesic $\gamma \in  W^{1,2} \left( [0,1];M\right)$ joining $x$ to $y$ be fixed. Since $\gamma$ minimizes the distance between $x$ and $y$ it cannot have self-intersection. Hence $(\Delta, g)$ admits an orthonormal frame along $\gamma$. There is an open neighborhood $\mathcal{V}$ of $\gamma([0,1])$ in $M$ and an orthonormal family $\mathcal{F}$ (with respect to $g$) of $m$ smooth vector fields  $X^1, \ldots, X^m$ such that
$$
\Delta (z) = \mbox{Span} \Bigl\{ X^1(z), \ldots, X^m(z) \Bigr\} \qquad \forall z \in \mathcal{V}.
$$
Moreover there is a control $u^{\gamma} \in L^2\left( [0,1];\R^m\right)$ such that
$$
\dot{\gamma} (t) = \sum_{i=1}^m u_i^{\gamma}(t) X^i\bigl( \gamma(t)\bigr) \qquad \mbox{a.e. } t \in [0,1].
$$
The End-Point mapping from $x$ is defined in an open neighborhood $\mathcal{U} \subset L^2\left( [0,1];\R^m\right)$ of $u^{\gamma}$ as
$$
\begin{array}{rcl}
    E^{x,1}_{\mathcal{F}} \, : \, \mathcal{U} & \longrightarrow & M \\
 u  & \longmapsto &  \gamma_u(1),
\end{array}
$$
where $\gamma_u$ is solution to the non-autonomous Cauchy problem
$$
\dot{\gamma}_u (t) = \sum_{i=1}^m u_i(t) X^i\bigl( \gamma_u(t)\bigr) \quad \mbox{a.e. } t \in [0,1], \quad \gamma_u(0)=x.
$$
There is locally a one-to-one correspondence between the set of horizontal paths starting from $x$ and the set of controls in $L^2\left( [0,1];\R^m\right)$. The End-Point mapping $E^{x,1}_{\mathcal{F}}$ is smooth in $\mathcal{U}$. 

\begin{definition}
A minimizing geodesic $\gamma$ is called singular if $E_{\mathcal{F}}^{x,1}$ is not a submersion at $u^{\gamma}$, that is if 
$$
D_{u^{\gamma}} E^{x,1}_{\mathcal{F}} \, : \, L^2\left( [0,1];\R^m\right) \longrightarrow T_{\gamma(1)}M
$$
is not onto.
\end{definition}

If $\gamma$ is a minimizing geodesic between $x$ and $y$ which is not singular, then it is the projection of what one calls a normal extremal, that is a trajectory of the Hamiltonian system associated canonically with $H:T^*M \rightarrow \R$ defined by
$$
H(x,p) := \frac{1}{2} \sum_{i=1}^m \bigl(  p \cdot X^i(x) \bigr)^2 \qquad \forall (x,p) \in T^*M \simeq \mathcal{V} \times \bigl( \R^n\bigr)^*.
$$
As a consequence any non-singular minimizing geodesic is smooth. (In the following definition, a geodesic is called non-trivial if it is not constant.)

\begin{definition}
Let $(\Delta,g)$ be a sub-Riemannian structure of rank $m$ on $M$. It is called ideal if it is complete and has no non-trivial singular minimizing geodesics.
\end{definition}

As explained in \cite{riffordbook}, ideal SR structures share the same properties as Riemannian manifolds outside the diagonal (the subset of $M\times M$ consisting of pairs $(x,x)$ with $x\in M$). Given a sub-Riemannian structure on $M$ and $x\in M$, we define the SR cut-locus at $x$ as 
\begin{eqnarray}\label{cutSR}
\mbox{cut}_{SR}(x) := \overline{\Sigma \left(d_{SR}(x,\cdot) \right) },
\end{eqnarray}
where $\Sigma \left(d_{SR}(x,\cdot) \right) $ denotes the set of $y\in M$ such that the pointed distance $d_{SR}(x,\cdot)$ is not differentiable at $y$. 

\begin{proposition}\label{PROPideal}
Let $(\Delta,g)$ be an ideal sub-Riemannian structure on $M$. Then the following properties hold:
\begin{itemize}
\item[(i)] The sub-Riemannian distance $d_{SR}$ is locally semiconcave outside the diagonal in $M\times M$, that is it can be written in local coordinates as the sum of a concave and a smooth function;
\item[(ii)] for every $x\in M$, the set $\mbox{cut}_{SR} (x)$ has Lebesgue measure zero;
\item[(iii)] for every $x\in M$, the set $\mbox{cut}_{SR} (x)\setminus \{x\}$ is exactly the closure of the set of $y \in M$ which can be joined to $x$ with at least two minimizing geodesics; 
\item[(iv)] for every $x\in M$, the pointed distance $d_{SR}(x,\cdot)$ is smooth on $M\setminus \cut_{SR}(x)$;
\item[(v)] for every $y\in M\setminus \{x\}$ and every minimizing geodesic from $x$ to $y$, we have 
$$
\gamma(t) \notin \cut_{SR}(x) \qquad \forall t \in (0,1).
$$
\end{itemize}
\end{proposition}

The first example of ideal sub-Riemannian structure is given by the Heisenberg group or more generally by fat distributions (see Section \ref{SECfat}). 

\subsection{Carnot groups}

We recall here basic facts on Carnot groups. We refer the reader to \cite{ledonne10} and references therein for further details. A Carnot group $(\G, \star )$ of step $s$ is a simply connected Lie group whose Lie algebra $\mathfrak{g}=T_0\G$ (we denote by $0$ the identity element of $\G$) admits a nilpotent stratification of step $s$, {\it i.e.} 
\begin{eqnarray}\label{carnot0}
\mathfrak{g}  = V_1 \oplus \cdots \oplus V_s,
\end{eqnarray}
with
\begin{eqnarray}\label{carnot1}
\bigl[ V_1, V_j\bigr] = V_{j+1} \quad \forall 1\leq j \leq s, \quad V_s \neq\{0\}, \quad V_{s+1}= \{0\}.
\end{eqnarray}
We denote by $m_j$ the dimension of each layer $V_j$ and by
\begin{eqnarray}\label{growth}
n= m_1 + \cdots + m_s,
\end{eqnarray}
the topological dimension of $\G$. Given $x\in \G$, the mapping $L_x :\G \rightarrow \G$ defined by 
$$
L_x(y) = x  \star  y \qquad \forall y \in \G,
$$
denotes the left-translation by the element $x$. Every $v\in \mathfrak{g} =T_0\G$ gives rise to a left-invariant vector field $X_v$ on $\G$ defined by 
$$
X_v(x) := \left(L_x\right)_* (v) := d_0L_x (v) \qquad \forall x \in \G.
$$
The exponential map 
$$
\exp_{\G}:\mathfrak{g} \rightarrow \G
$$
evaluated at $v\in \mathfrak{g}$ is given by $\gamma(1)$, where $\gamma :\R \rightarrow \G$ is solution to the Cauchy problem $\dot{\gamma}=X_v(\gamma), \, \gamma(0)=0$. By simple-connectedness of $\G$ and nilpotency of $\mathfrak{g}$, $\exp_{\G}$ is a smooth diffeomorphism, which allows to identify $\G$ with its Lie algebra $\mathfrak{g} \simeq \R^n$. Fix a vector basis $v_1, \ldots, v_n$ of $\mathfrak{g}$ satisfying  
$$
\left\{
\begin{array}{l}
v_1, \ldots, v_{m_1} \in V_1,\\
v_{m_1+1}, \ldots, v_{m_1+m_2} \in V_2,\\
\qquad \vdots\\
  v_{n-m_s+1}, \ldots, v_n \in V_s.
\end{array}
\right.
$$
The coordinates $(x_1, \ldots, x_n)$ are the exponential coordinates of $\exp_{\G} \left( \sum_{i=1}^m x_iv_i\right) \in \G$. The group law on $\G$ can be pulled back into a group law on $\R^n$ (that is still denoted by $\star$) by exponential coordinates,
$$
x \star y = \exp_{\G}^{-1} \left( \exp_{\G} \left(\sum_{i=1}^n x_i v_i\right) \star \exp_{\G} \left( \sum_{i=1}^n y_i v_i \right) \right) \qquad \forall x,y \in \R^n.
$$
With such a group law $\R^n$ is a Lie group whose Lie algebra is isomorphic to $\mathfrak{g}$, making $(\G,\star)$ and $(\R^n, \star)$ isomorphic. If $i\in \{1,\ldots,n\}$ is an index such that 
$$
m_1 + \ldots + m_{d_i-1} < i \leq m_1 + \ldots + m_{d_i}
$$
for some $d_i \in \{1, \ldots,s\}$, the coordinate $x_i$ will be said to have degree $d_i$. The family $\{\delta_{\lambda}\}_{\lambda >0}$ defined as 
$$
\delta_{\lambda} \left(x_1,\ldots, x_n\right) = \left( \lambda^{d_1} x_1,  \lambda^{d_2} x_2, \ldots, \lambda^{d_n} x_n \right)  \qquad \forall x \in \R^n,
$$ 
provides  a one-parameter family of dilations in $\R^n$. We have 
$$
\delta_1=\mbox{Id}_{\R^n} \quad \mbox{and} \quad \delta_{\lambda \lambda'}= \delta_{\lambda} \circ \delta_{\lambda'} \, \forall \lambda, \lambda'>0.
$$
For every $i=1, \ldots,n$, denote by $X^i$  the left-invariant vector field on $\R^n$, written as
$$
X^i(x) = \sum_{k=1}^n a_{ki} (x) \, \partial_k,
$$
such that $X^i(0)= \partial_i$. It can be checked that each $X^i$ is homogeneous of degree $d_i$ with respect to $\{\delta_{\lambda} \}_{\lambda >0}$, that is 
$$
a_{ki}\left(\delta_{\lambda}(x)\right) = \lambda^{d_k-d_i} a_{ki}(x) \qquad \forall k=1, \ldots,n, \, \forall x \in \R^n.
$$
In particular, each vector field $X^1, \ldots, X^{m_1}$ satisfies 
\begin{eqnarray}\label{dilation}
X^i \left( \delta_{\lambda}(x)\right) = \lambda^{-1} \, \delta_{\lambda} \left( X^i(x)\right) \qquad \forall x \in \R^n.
\end{eqnarray}
The quantity 
$$
D := \sum_{i=1}^n d_i = \sum_{j=1}^s j m_j
$$
is called the homogeneous dimension of $\G$. The Haar measure of $(\R^n,\star)$ is (up to a multiplicative constant) the Lebesgue measure $\mathcal{L}^n$. Any metric on the first layer provides a sub-Riemannian metric by translation, and subsequently gives  a sub-Riemannian structure $(\Delta,g)$ (with $\Delta(x)= (L_x)_* (V_1)$ and $g_x=(L_x)_*(g)$) on $\G$. Denote by $d_{SR}$ the sub-Riemannian distance for the left-invariant SR structure $(\Delta,g)$. The metric space $(\G,d_{SR})$ is necessarily complete and we check easily that if $\gamma :[0,1] \rightarrow \G$ is a minimizing geodesic from $x:=\gamma(0)$ to $y:=\gamma(1)$, then for every $z\in \G$, the horizontal path
$$
t \in [0,1] \, \longmapsto z \star \gamma(t)
$$
is a minimizing geodesic from $z\star x$ to $z\star y$. In particular, the sub-Riemannian distance is invariant by left-translations,
$$
d_{SR} = \left( z\star x, z\star y\right) = d_{SR} (x,y) \qquad \forall x,y, z \in \G.
$$
If we pull-back everything in $\R^n$ by exponential coordinates, then the homogeneity of the first layer (see (\ref{dilation})) implies that if $\gamma :[0,1] \rightarrow \G$ is a minimizing geodesic from $0$ to $x:=\gamma(1)$, then for every $\lambda>0$, the horizontal path $\delta_{\lambda}\circ \gamma$ is minimizing from $0$ to $x$. Then we have 
\begin{eqnarray}\label{dilation1}
d_{SR} \left(0,\delta_{\lambda}(x)\right) = \lambda \, d_{SR}(0,x) \qquad \forall x \in \R^n.
\end{eqnarray}
In particular, the homogeneity property of the SR distance yields the invariance of sub-Riemannian balls by dilations, 
\begin{eqnarray}\label{dilation2}
\delta_{\lambda} \bigl( B_{SR}(0,r)\bigr) = B_{SR} \bigl(0,\lambda r\bigr) \qquad \forall \lambda, r>0,
\end{eqnarray}
where $B_{SR}(0,\lambda)$ denotes the sub-Riemannian ball centered at the origin with radius $\lambda$.

\subsection{Sub-Riemannian Euler-Arnold equations}\label{SECEA}

In the spirit of Arnold \cite{arnold66}, we can write the equations of geodesics in the Lie algebra $\mathfrak{g}$. Recall that $m_l=:m(l)$ denotes the dimension of each layer $V_l$. For each $l=1, \ldots, s$, we pick an orthonormal basis  $e_1^l, \ldots, e_{m_l}^l$ of $V_l$ (remember (\ref{carnot0})). Thus by (\ref{carnot1}), there are structure constants $\left( \alpha_{ij,k}^{1,l}\right)$ such that
\begin{eqnarray}\label{structure}
\left[ e_{i}^{1} , e_j^{l}\right] = \sum_{k=1}^{m(l+1)} \alpha_{ij,k}^{1l} \, e_k^{l+1}.
\end{eqnarray} 
For every $l=1, \ldots, s-1$ and every $k=1, \ldots, m(l+1)$  denote by $A_{k}^{1l}$ the $m(l) \times m$ matrix whose the coefficients are given by
\begin{eqnarray}\label{structure2}
\left(A_k^{1l}\right)_{ij} = \alpha_{ij,k}^{1l}.
\end{eqnarray}
For each $i,l$, we denote by $X_i^{l}$ the left-invariant vector field on $\G$ obtained from $e_i^{l}$. Let us now consider a minimizing geodesic $\gamma :[0,1] \rightarrow \G$ starting at the origin. It is associated with a smooth control $u^{\gamma} : [0,1] \rightarrow \R^m$ (note that $m_1=m$) satisfying
 $$
\dot{\gamma} (t) = \sum_{i=1}^m u_i^{\gamma}(t) X_i^1 \bigl( \gamma(t)\bigr) \quad \forall t \in [0,1], \quad \gamma(0)=0.
$$
The following result follows easily from the Hamiltonian equation for SR geodesics.

\begin{proposition}\label{SREA}
The horizontal curve $\gamma$ is the projection of a normal extremal if and only if there are $s$ smooth functions 
$$
h^{1} : [0,1] \rightarrow \R^{m_1}, \quad \ldots, \quad h^s: [0,1] \rightarrow \R^{m_s}
$$
satisfying 
\begin{eqnarray}\label{SYSEA}
\left\{
\begin{array}{rcl}
\dot{h}^{l} & = & \sum_{k=1}^{m(l+1)}  h_k^{l+1}\, A_k^{1l} \, h^1 \quad \forall l\in 1, \ldots, s-1,\\
\dot{h}^s & = & 0,
\end{array}
\right.
\end{eqnarray}
such that
\begin{eqnarray}\label{uEA}
u_i^{\gamma}(t) = h_i^1(t) \qquad \forall i=1, \ldots, m, \, \forall t \in [0,1].
\end{eqnarray}
\end{proposition}

\begin{proof}
Define in local coordinates along $\gamma([0,1])$, the Hamiltonian $H: \mathcal{V} \times (\R^n)^* \rightarrow \R$ by 
\begin{equation}
\label{normal0}
H(x,p) := \frac{1}{2} \sum_{i=1}^m \bigl(  p \cdot X_i^1(x) \bigr)^2
\end{equation}
for all $(x,p) \in \mathcal{V} \times (\R^n)^*$. Then there is a smooth arc $p:[0,1] \longrightarrow (\R^n)^*$ such that the pair $(\gamma,p)$ satisfies
\begin{equation}
\label{normal1}
\left\{ \begin{array}{rcl}
\dot{\gamma}(t) & = & \frac{\partial H}{\partial p} (\gamma(t),p(t)) = \sum_{i=1}^m \left[p(t) \cdot X_i^1(\gamma(t)) \right]  X^i(\gamma(t)) \\
\dot{p}(t) & = & -\frac{\partial H}{\partial x} (\gamma(t),p(t)) = - \sum_{i=1}^m \left[p(t) \cdot X_i^1(\gamma(t))\right] p(t) \cdot d_{\gamma(t)} X_i^1
\end{array}
\right.
\end{equation}
for any $t \in [0,1]$ and
\begin{equation}
\label{normal2}
u^{\gamma}_i(t) = p(t) \cdot X^1_i(\gamma(t))  \qquad \forall t\in [0,1], \quad\forall i=1,\ldots, m.
\end{equation}
Setting $h_i^l:= p(t) \cdot X_i^l(\gamma(t))$, we have (with the convention $[X,Y](x) = d_xX (Y(x)) - d_xY (X(x))$)
\begin{eqnarray*}
\dot{h}_i^l(t) & = & \dot{p}(t) \cdot X_i^l(\gamma(t)) + p(t) \cdot d_{\gamma(t)} X_i^l \bigl(\dot{\gamma}(t)\bigr) \\
& = & - \sum_{j=1}^m  h_j^1(t)   p(t) \cdot d_{\gamma(t)} X_j^1 \left(X_i^l(\gamma(t))\right) +  p(t) \cdot d_{\gamma(t)} X_i^l \left( \sum_{j=1}^m h_j^1(t) X_j^1(\gamma(t)) \right) \\
& = & \sum_{j=1}^m  h_j^1(t)   p(t) \cdot \left[ X_i^1,X_j^l\right] (\gamma(t)) =  \sum_{j=1}^m h_j^1(t) p(t)  \cdot \left[ e_i^1,e_j^l\right]. 
\end{eqnarray*}  
We conclude by (\ref{structure})-(\ref{structure2}).
\end{proof}

Let us now pull back $\gamma_u :[0,1] \rightarrow \G$ by the $\G$-exponential in order to obtain a curve $c_{u}:[0,1] \rightarrow  \mathfrak{g}$. Recall that for every $v \in \mathfrak{g}$, the family of linear maps $\{\mbox{ad}_v^k\}_{k\in \N} :  \mathfrak{g} \rightarrow  \mathfrak{g}$ is defined by
$$
\mbox{ad}_v^0 (w)= w, \quad \mbox{ad}_v^{k+1}(w) = \mbox{ad}^1_v \left( \mbox{ad}_v^k(w) \right) = \left[ v, \mbox{ad}_v^k (w)\right] \qquad \forall w \in \mathfrak{g}.
$$
We have

\begin{lemma}
The smooth curve $c_{u}:[0,1] \rightarrow  \mathfrak{g}$ satisfies
\begin{eqnarray}\label{eqE}
\sum_{k=0}^{s-1} \frac{(-1)^k}{(k+1)!} \ad_{c_u(t)}^k \left( \dot{c}_u(t)\right) = \sum_{i=1}^m h_i^1(t) e_i^1 \qquad \forall t \in [0,1].
\end{eqnarray}
\end{lemma} 

\begin{proof}
As a simply connected nilpotent Lie group, we may view $\G$ as a closed subgroup of the group of upper triangular matrices of a certain size having $1$'s on the diagonal. In that case, the exponential is given by the usual exponential on matrices and the adjoint map is defined by $\mbox{ad}_A(B) = [A,B] = AB-BA$ for any $A,B$. Given two squared matrices $A,B$ of the same size, we first show that 
\begin{eqnarray}\label{matrices}
F(A,B) :=     \frac{d}{dt} \Bigl\{ e^{-A}  e^{A+tB} \Bigr\}_{t=0}    = e^{-A} \frac{d}{dt} \Bigl\{ e^{A+tB} \Bigr\}_{t=0} = \sum_{k=0}^{\infty} \frac{(-1)^k}{(k+1)!} \, \mbox{ad}_A^k (B).  
\end{eqnarray}
As a matter of fact, on the one hand we have 
\begin{eqnarray}\label{rio1}
F(A,B) & = & \left( \sum_{k=0}^{\infty} \frac{(-1)^k}{k!} \, A^k \right) \left( \sum_{l=1}^{\infty} \frac{1}{l!} \, \frac{d}{dt}  \Bigl\{ (A+tB)^l \Bigr\}_{t=0} \right)\nonumber \\
& = &  \left( \sum_{k=0}^{\infty} \frac{(-1)^k}{k!} \, A^k \right) \left( \sum_{l=1}^{\infty}\frac{1}{l!} \sum_{q=1}^l A^{l-q} B A^{q-1} \right)\nonumber\\
& =&  \left( \sum_{k=0}^{\infty} \frac{(-1)^k}{k!} \, A^k \right) \left( \sum_{l=0}^{\infty}  \sum_{r=0}^l \frac{1}{(l+1)!} \, A^{l-r} B A^{r} \right)\nonumber\\
& = & \sum_{k=0}^{\infty} \sum_{l=0}^{\infty} \sum_{s=0}^l \frac{(-1)^k}{k! (l+1)!} \, A^k A^s B A^{l-s} \nonumber\\
& = &   \sum_{k=0}^{\infty}  \sum_{s=0}^{\infty} \sum_{j=k+s}^{\infty} \frac{(-1)^k}{k! (j+1-k)!} \, A^{k+s} B A^{j-k-s} \nonumber\\
& = & \sum_{j=0}^{\infty}   \sum_{i=0}^{j} \sum_{r=0}^{i} \frac{(-1)^r}{(j+1)!} \left( \begin{array}{c}j+1\\ r\end{array}\right)\, A^{i} B A^{j-i} \nonumber\\
& = &  \sum_{j=0}^{\infty}   \sum_{i=0}^{j} \frac{1}{(j+1)!} \, \left( \sum_{r=0}^{i} (-1)^r \left( \begin{array}{c}j+1\\ r\end{array}\right) \right)\, A^{i} B A^{j-i}\nonumber \\
& = &  \sum_{j=0}^{\infty}   \sum_{i=0}^{j} \frac{1}{(j+1)!} \, \left((-1)^i \left( \begin{array}{c}j\\ i\end{array}\right) \right)\, A^{i} B A^{j-i} \nonumber\\
& = &  \sum_{j=0}^{\infty}  \frac{(-1)^j}{(j+1)!}  \left(  \sum_{i=0}^{j}  \left( \begin{array}{c}j\\ i\end{array}\right) \,  (-1)^{j-i} A^{i} B A^{j-i} \right),
\end{eqnarray}
where we used the identity 
$$
\sum_{r=0}^{i} (-1)^r \left( \begin{array}{c}j+1\\ r\end{array}\right) =  (-1)^i \left( \begin{array}{c}j\\ i\end{array}\right).
$$
On the other hand, the binomial theorem yields
$$
\mbox{ad}_A^k(B) = \sum_{r=0}^k \left( \begin{array}{c}k\\ r\end{array}\right) \,(-1)^{k-r}  A^r B A^{k-r}.
$$ 
Then the identity  (\ref{matrices}) follows from  the above formula and (\ref{rio1}). By construction, the path $c_{u}:[0,1] \rightarrow \G$ satisfies for every $t\in [0,1]$,
$$
d_{c_{u}(t)}\exp_{\G} \left( \dot{c}_{u}(t)\right) = \dot{\gamma}_u(t) = \sum_{i=1}^m u_i(t) X^1_i \left( \gamma_u(t)\right),
$$ 
which yields
\begin{eqnarray}\label{rio2}
d_{\gamma_u(t)} L_{\gamma_u(t)^{-1}} \Bigl( d_{c_{u}(t)}\exp_{\G} \left( \dot{c}_{u}(t)\right)\Bigr) &= & d_{\gamma_u(t)} L_{\gamma_u(t)^{-1}} \left( \sum_{i=1}^m u_i(t) X^1_i \left( \gamma_u(t)\right)\right) \nonumber \\
& = & \sum_{i=1}^m h_i^1(t) e_i^1.
\end{eqnarray}
But there holds for every $t\in [0,1]$,
\begin{eqnarray*}
\frac{d}{ds} \Bigl\{ \exp_{\G} \left( -\gamma_u(t)\right) \star \exp_{\G} \left( c_{u}(t+s)\right) \Bigr\}_{s=0} & = & \frac{d}{ds} \Bigl\{ L_{\gamma_u(t)^{-1}} \Bigl( \exp_{\G} \left( c_{u}(t+s)\right) \Bigr) \Bigr\}_{s=0} \\
& = & d_{\gamma_u(t)} L_{\gamma_u(t)^{-1}} \Bigl( d_{c_{u}(t)}\exp_{\G} \left( \dot{c}_{u}(t)\right)\Bigr).
\end{eqnarray*}
We conclude by (\ref{matrices}), (\ref{rio2}) and the fact that $\G$ is nilpotent of step $s$.
\end{proof}

\section{Proof of the results}\label{SECproof}

\subsection{Proof of Theorem \ref{THM1}}

Up to pulling-back the metric to $\R^n$ by the exponential map $\exp_{\G}$, we can assume that our Carnot group is $(\R^n,\star)$ equipped with a left-invariant metric $g$ and with  the Lebesgue measure  $\mathcal{L}^n$ as Haar measure.  We need to show that there is $N>0$ such that for every measurable set $A$ with $0< \mathcal{L}^n(A)<\infty$, 
\begin{eqnarray}\label{ineqreq}
\vol_{\G} (A_s) \geq s^N \, \vol_{\G} (A) \qquad \forall s \in [0,1].
\end{eqnarray}
Note that since the sub-Riemannian structure is invariant by translation, it is sufficient to prove the result for $x=0$. Let us now assume that the sub-Riemannian structure on $\G$ is ideal. 

\begin{lemma}\label{LEM1}
There is $N>0$ such that for every measurable set $A \subset B_{SR}(0,1) \setminus B_{SR}(0,1/2)$, we have
\begin{eqnarray}\label{MCPrio}
\vol_{\G} \bigl(A_{s}\bigr) \geq s^{N} \vol_{\G} (A) \qquad \forall s \in [1/2,1].
\end{eqnarray}
\end{lemma}

\begin{proof}[Proof of Lemma \ref{LEM1}]
By Proposition \ref{PROPideal}, the pointed distance function $f : x \mapsto d_{SR}(0,x)$ is locally semiconcave outside the origin and smooth outside $\cut_{SR}(0)$. Then there is a constant $K>0$ (see \cite{cs04}) such that 
\begin{eqnarray}\label{Hf}
\mbox{Hess}_x f \leq K I_n,
\end{eqnarray}
for every $x$ in the open set $\Omega:=  \left( \R^n \setminus \cut_{SR}(0)\right) \cap \left( B_{SR}(0,1) \setminus B_{SR}(0,1/4)\right)$ (here $\mbox{Hess}$ denotes the canonical Hessian in $\R^n$). Denote by $Z: \Omega \rightarrow \R^n$ the optimal synthesis associated with $f=d_{SR}(0,\cdot)$, that is the smooth vector field defined by 
\begin{eqnarray}\label{defX}
Z(x) := - \mathcal{P}_{\Delta(x)} \left( d_x f\right) \qquad \forall x \in \Omega,
\end{eqnarray}
where $\mathcal{P}_{\Delta(x)}  : \left( \R^n\right)^* \rightarrow \R^n$ denotes the projection to $\Delta(x)$ which is defined by
$$
\mathcal{P}_{\Delta(x)} (p) := v \in \Delta(x) \mbox{ such that } p\cdot w = g(v,w) \, \forall w \in \Delta(x).
$$
Note that since $f=d_{SR}(0,\cdot)$ is solution to the horizontal eikonal equation 
$$
H\bigl(x,d_xf\bigr) =\frac{1}{2} \qquad \forall x \in \R^n \setminus \{0\},
$$
the vector field $Z$ has always norm $1$ with respect to $g$. We denote by $\phi_t^Z$ the flow of $Z$. Set $\mathcal{A}:= B_{SR}(0,1) \setminus B_{SR}(0,1/2)$ and fix a measurable set $A\subset \mathcal{A}$.  Note that 
$$
A_s = \phi^Z_{1-s} (A_1) \subset \Omega \qquad \forall s \in [1/2,1].
$$
Then we have by definition of the divergence,
$$
\frac{d}{dt} \Bigl\{ \mathcal{L}^n \bigl(\phi^Z_t \bigl(A_{1}\bigr) \bigr)\Bigr\} = \int_{\phi^Z_t(A_1)} \mbox{div}_x  Z \, dx \qquad \forall t \in (0,1/2].
$$
Then writing (\ref{defX}) as $Z(x) = -G\left(x,d_x f\right)$ and setting  $G^p:=G(\cdot,p), G^x := G(x,\cdot)$, we have (note that $G^x$ is linear)
$$
\mbox{div}_x Z = \sum_{i=1}^n \frac{\partial Z_i}{\partial x_i} (x) = - \mbox{div}_x G^{d_x f} - \mbox{tr} \left( G^x \circ \mbox{Hess}_x f\right) \qquad \forall x \in \Omega.
$$
The first term in the above formula is bounded ($f$ is Lipschitz on $\Omega$ and $G$ is smooth) and by (\ref{Hf}) the second term is bounded from below ($G^x$ is a linear projection). This shows that there is $C>0$ such that for every $t\in (0,1/2]$,
$$
\frac{d}{dt} \Bigl\{ \mathcal{L}^n \bigl( A_{1-t}\bigr) \Bigr\}=  \frac{d}{dt} \Bigl\{ \mathcal{L}^n \bigl(\phi^X_t \bigl(A_{1}\bigr) \bigr)\Bigr\}  \geq  \int_{\phi^X_t(A_1)} - C dz \geq - C \mathcal{L}^n \bigl(A_{1-t}\bigr).
$$
By Gronwall's Lemma, we infer that for every $s\in [1/2,1)$,
$$
\mathcal{L}^n \bigl(A_s\bigr) \geq e^{C(s-1)} \mathcal{L}^n \bigl(A_1\bigr).
$$
We conclude easily.
\end{proof}

\begin{lemma}\label{LEM2}
There is $N>0$ such that for every $k\in \Z$ and every measurable set 
\begin{eqnarray}\label{assumptionrio}
A \subset B_{SR}\left(0,\frac{1}{2^k}\right) \setminus B_{SR}\left(0,\frac{1}{2^{k+1}}\right),
\end{eqnarray}
we have
\begin{eqnarray}\label{MCPrio2}
\vol_{\G} (A_s) \geq s^{N} \, \vol_{\G} (A) \qquad \forall s \in [0,1].
\end{eqnarray}
\end{lemma}

\begin{proof}[Proof of Lemma \ref{LEM2}]
Recall that for every $\lambda>0$, $\delta_{\lambda}$ denotes the dilations in $\R^n$ of ratio $\lambda$. By dilations properties (\ref{dilation1})-(\ref{dilation2}), for every integer $k$ and every measurable set $A$ satisfying (\ref{assumptionrio}), we have 
$$
 \delta_{2^k}(A) \subset B_{SR}(0,1) \setminus B_{SR}(0,1/2) \quad \mbox{and} \quad \delta_{2^k}\bigl(A_s\bigr) = \left( \delta_{2^k}(A)\right)_s \, \forall s \in [0,1].
$$
As a consequence, it is sufficient to prove (\ref{MCPrio2}) with a measurable set $A$ satisfying (\ref{assumptionrio}) with $k=0$.  
Note that for every $s\in (0,1/2)$, $A_{s} = \left(A_{2s}\right)_{1/2}$. Given $s\in (1/4,1/2)$, we denote by $\delta_{(2s)^{-1}}$ the dilation of ratio $1/(2s)$ and we set $$
B:=\delta_{(2s)^{-1}} \left(A_{2s}\right) \subset B_{SR}(0,1) \setminus B_{SR}(0,1/2).
$$
By dilation properties (see (\ref{dilation1})), $B_{1/2} = \delta_{(2s)^{-1}} \left(A_{s}\right)$. Then using (\ref{MCPrio}), we get (recall that $D$ denotes the homogeneous dimension)
$$
(2s)^{-D} \vol_{\G}\bigl( A_{s}\bigr) = \vol_{\G} \bigl( B_{1/2}\bigr) \geq \frac{1}{2^{N}} \vol_{\G} (B) = \frac{(2s)^{-D}}{2^{N}} \vol_{\G} \bigl(A_{2s}\bigr).
$$
Thus we obtain recursively (we set $\kappa := 2^{-N}$)
$$
\mbox{vol} \bigl(A_{1/2^l}\bigr) \geq \kappa^k \mbox{vol} (A).
$$
Which implies for every $s\in [0,1]$ with $2^l s \in [1/2,1]$ (again using  (\ref{MCPrio})),
$$
\vol_{\G} \bigl( A_s\bigr) = \vol_{\G} \left( \bigl( A_{2^ls}\bigr)_{1/2^l}\right) \geq \kappa^l \vol_{\G} \bigl(  A_{2^ls}\bigr) \geq  \kappa^l 2^{lN} s^{N} \vol_{\G}(A) = s^{N} \vol_{\G}(A).
$$
This concludes the proof.
\end{proof}

We conclude easily the proof of the ideal case by decomposing any measurable set $A\subset \R^n$ with  $0< \mathcal{L}^n(A)<\infty$ as 
$$
A = \bigcup_{k\in \Z}  A^k \quad \mbox{with} \quad A^k := A \cup \left( B_{SR}\left(0,\frac{1}{2^k}\right) \setminus B_{SR}\left(0,\frac{1}{2^{k+1}}\right)\right)
$$
and applying Lemma \ref{LEM2}.

\subsection{Proof of Theorem \ref{THM2}}

For every $h\in \R^n$, the solution $h=h_{h}:[0,1] \rightarrow \R^n$ of the system (\ref{SYSEA}) starting at $h$ gives the control law $u=u_{h} :[0,1] \rightarrow \R^m$ of the normal extremal starting at the origin with covector $p= \sum_{i=1}^n h_i dv_i$ (where $dv_1, \ldots, dv_n$ denotes the dual basis of $v_1, \ldots, v_n$). Recall that the pull back of the corresponding geodesic, which is denoted by $c= c_{h}:[0,1] \rightarrow  \mathfrak{g}$  satisfies (see (\ref{eqE})) 
\begin{eqnarray}\label{riolast} 
\sum_{k=0}^{s-1} \frac{(-1)^k}{(k+1)!} \mbox{ad}_{c(t)}^k \left( \dot{c}(t)\right) = \sum_{i=1}^m h_i(t) e_i^1 \qquad \forall t\in [0,1].
\end{eqnarray}
For every integer $l\geq 0$, we denote by $c^{(l)}(0)$ the $l$-th derivative of $c$ at $t=0$ and we set $V_l=\{0\}$ for $l\geq s+1$.

\begin{lemma}\label{LEMl}
For every $h \in \R^n$,
$$
\dot{c}(0) \in V_1 \quad \mbox{and} \quad   
c^{(l)}(0) \in V_1  \oplus \cdots \oplus V_{l-1} \, \forall l \geq 2.
$$
\end{lemma}

\begin{proof}[Proof of Lemma \ref{LEMl}]
By bilinearity of the Lie bracket in $ \mathfrak{g}$ and the binomial theorem, we check easily that for every integer $k\geq 0$, and for every integer $l\geq 1$,
\begin{eqnarray*}
\frac{d^l}{dt^l} \Bigl\{  \mbox{ad}_{c(t)}^{k+1} \left( \dot{c}(t)\right) \Bigr\}_{t=0} & = & \frac{d^l}{dt^l} \Bigl\{ \bigl[c(t), \mbox{ad}^{k}_{c(t)} \left( \dot{c}(t)\right) \bigr]\Bigr\}_{t=0} \\
& = & \sum_{r=0}^l \left( \begin{array}{c}l\\ r\end{array}\right) \, \left[c^{(l-r)}(0),  \frac{d^r}{dt^r} \Bigl\{  \mbox{ad}_{c(t)}^{k} \left( \dot{c}(t)\right) \Bigr\}_{t=0} \right] \\
& = &  \sum_{r=0}^{l-1} \left( \begin{array}{c}l\\ r\end{array}\right) \, \left[c^{(l-r)}(0),  \frac{d^r}{dt^r} \Bigl\{  \mbox{ad}_{c(t)}^{k} \left( \dot{c}(t)\right) \Bigr\}_{t=0} \right], 
\end{eqnarray*}
since $c^{(0)}(0)=c(0)=0$. Then by induction we infer that 
$$
\frac{d^l}{dt^l} \Bigl\{  \mbox{ad}_{c_u(t)}^{k} \left( \dot{c}_u(t)\right) \Bigr\}_{t=0} =0 \qquad \forall k \geq 1, \, \forall l \in [0,k] ,
$$
which implies for any $l\geq k+2\geq 2$,
\begin{eqnarray}\label{rioairport}
\frac{d^l}{dt^l} \Bigl\{  \mbox{ad}_{c(t)}^{k+1} \left( \dot{c}(t)\right) \Bigr\}_{t=0} =  \sum_{r=k+1}^{l-1} \left( \begin{array}{c}l\\ r\end{array}\right) \, \left[c^{(l-r)}(0),  \frac{d^r}{dt^r} \Bigl\{  \mbox{ad}_{c(t)}^{k} \left( \dot{c}(t)\right) \Bigr\}_{t=0} \right].
\end{eqnarray}
We claim that 
\begin{eqnarray}\label{nat}
\frac{d^{k+r}}{dt^{k+r}} \Bigl\{  \mbox{ad}_{c(t)}^{k} \left( \dot{c}(t)\right) \Bigr\}_{t=0} \in V_1  \oplus \cdots \oplus V_{k+r} \qquad \forall k \geq 0, \, \forall r \geq 1,
\end{eqnarray}
which gives the desired result for $l=0$ and $r\geq 1$. Let us prove it by induction. Taking $t=0$ in (\ref{riolast}) and its derivative gives 
$$
\dot{c}(0)  =  \sum_{i=1}^m h_i(0) e_i^1, \quad \ddot{c}(0) = \sum_{i=1}^m \dot{h}_i(0) e_i^1 \in V_1,
$$
whose the second equality means that (\ref{nat}) is satisfied with $k=0, r=1$. Applying (\ref{rioairport}) with $l=k+2$ yields
$$
\frac{d^{k+2}}{dt^{k+2}} \Bigl\{  \mbox{ad}_{c(t)}^{k+1} \left( \dot{c}(t)\right) \Bigr\}_{t=0} = (k+2) \, \left[\dot{c}(0),  \frac{d^{k+1}}{dt^{k+1}} \Bigl\{  \mbox{ad}_{c(t)}^{k} \left( \dot{c}(t)\right) \Bigr\}_{t=0} \right] \qquad \forall k\geq 0.
$$
We deduce easily that (\ref{nat}) holds for any pairs $(k,r)$ with $k\geq 0$ and $r=1$. Assume now that  (\ref{nat}) holds for any pairs $(k,r)$ with $k\geq 0$ and $r\leq q$ for some integer $q\geq 1$ and show how to deduce the result for the pairs $(k,q+1)$. Taking $q+1$ derivatives in  (\ref{riolast}) gives
\begin{eqnarray*}
c^{(q+2)}(0)  &= &   \sum_{i=1}^m h_i^{(q+1)}(0) e_i^1 - \sum_{k=1}^{s-1}  \frac{(-1)^k}{(k+1)!} \, \frac{d^{q+1}}{dt^{q+1}} \Bigl\{  \mbox{ad}_{c(t)}^{k} \left( \dot{c}(t)\right) \Bigr\}_{t=0} \\
& = &   \sum_{i=1}^m h_i^{(q+1)}(0) e_i^1 - \sum_{k=1}^{q}  \frac{(-1)^k}{(k+1)!} \, \frac{d^{q+1}}{dt^{q+1}} \Bigl\{  \mbox{ad}_{c(t)}^{k} \left( \dot{c}(t)\right) \Bigr\}_{t=0} \\
& = &  \sum_{i=1}^m h_i^{(q+1)}(0) e_i^1 - \sum_{k=1}^{q}  \frac{(-1)^k}{(k+1)!} \, \frac{d^{k+(q+1-k)}}{dt^{k+(q+1-k)}} \Bigl\{  \mbox{ad}_{c(t)}^{k} \left( \dot{c}(t)\right) \Bigr\}_{t=0}.
\end{eqnarray*}
Since $q+1-k \leq q$ for every $k\in [1,q]$, we infer that 
$$
c^{(q+2)}(0) = \frac{d^{q+1}}{dt^{q+1}} \Bigl\{  \mbox{ad}_{c(t)}^{0} \left( \dot{c}(t)\right) \Bigr\}_{t=0} \in V_1  \oplus \cdots \oplus V_{q+1}.
$$
Now applying (\ref{rioairport}) with $l=k+1+q+1$ and $k\geq 0$ yields
\begin{eqnarray*}
\frac{d^l}{dt^l} \Bigl\{  \mbox{ad}_{c(t)}^{k+1} \left( \dot{c}(t)\right) \Bigr\}_{t=0} & = & \sum_{r=k+1}^{k+1+q} \left( \begin{array}{c}l\\ r\end{array}\right) \, \left[c^{(l-r)}(0),  \frac{d^r}{dt^r} \Bigl\{  \mbox{ad}_{c(t)}^{k} \left( \dot{c}(t)\right) \Bigr\}_{t=0} \right] \\
& = & \sum_{r=1}^{q+1} \left( \begin{array}{c}l\\ k+r\end{array}\right) \, \left[c^{(q+2-r)}(0),  \frac{d^{k+r}}{dt^{k+r}} \Bigl\{  \mbox{ad}_{c(t)}^{k} \left( \dot{c}(t)\right) \Bigr\}_{t=0} \right] 
\end{eqnarray*}
We conclude easily by induction on $k$. 
\end{proof}

Let us finish the proof of Theorem \ref{THM2}. Thanks to a result by Agrachev  \cite{agrachev09}, the set of smooth points is open and dense in $\G$. This implies that there are $x\in \G$ and an open (and bounded) neighborhood $\mathcal{V}$ of $x$ such that for every $y\in \mathcal{V}$, there is a unique minimizing geodesic $\gamma_y: [0,1] \rightarrow \G$ from $0$ to $y$. Moreover, this geodesic is not a singular point of the sub-Riemannian exponential map $\exp_0 : T_0^*\G =\mathfrak{g}^* \rightarrow \G$; in particular $\gamma_y$ is the projection of a unique normal extremal. By analyticity, this implies that for every $t\in [0,1]$ small enough and every $y\in \mathcal{V}$, there is a unique minimizing geodesic from $0$ to $\gamma_y(t)$ which in addition is not a singular point of the SR exponential map (see \cite{riffordbook}). Let us pull back everything in $\mathfrak{g}$. Define the map $\Psi : \R^n \rightarrow \mathfrak{g}$ by (we now denote $c$ by $c_{h}$ to stress the dependence of $c$ upon $h$)
$$
\Psi (h) := c_{h} (1) \qquad  \forall h \in \R^n.
$$
By the above discussion, we have for every measurable set $A \subset \mathcal{V}$ and every $s\in [0,1]$ small enough,
$$
\mathcal{L} (A_1) = \int_{D_1}  D(h) \, dh \quad \mbox{and} \quad 
\mathcal{L} (A_s) = \int_{D_s}  D(h)  \, dh = \int_{D_1} s^n D(sh)  \, dh,
$$
where $D$ denotes the Jacobian determinant of $\Psi$ (which is nonnegative). Therefore if a sub-Riemannian structure on $\G$ makes it a measured metric space which is geodesic with negligeable cut loci satisfying $\MCP(0,N)$, then for every measurable set $A  \subset \R^n$ with $\Psi(A) \subset  \mathcal{V}$, there holds
$$
 \int_{A} s^n D(sh)  \, dh \geq s^N \int_{A}  D(h) \, dh \qquad \forall s \in [0,1] \mbox{ small},
$$
which implies
$$
D\bigl(sh_x\bigr) \geq s^{N-n} D\bigl(h_x\bigr)  \qquad \forall s \in [0,1],
$$
where $h_x\in \R^n$ is defined as $\Psi(h_x)=x$. Define the function $\tilde{\Psi} : [0,1] \times \R^n \rightarrow \mathfrak{g}$ by
$$
\tilde{\Psi} (s,h) =  \Psi(sh) \qquad \forall s \in [0,1], \, \forall h \in \R^n.
$$
It is smooth (it is indeed analytic) and  satisfies
\begin{eqnarray}\label{18dec1}
\det \left( \frac{\partial \tilde{\Psi}}{\partial h} \bigl(s,h_x)\bigr) \right) \geq s^N \det \left( \frac{\partial \tilde{\Psi}}{\partial h} \bigl(1,h_x)\bigr) \right) \qquad \forall s\in [0,1] \mbox{ small}.
\end{eqnarray}
Note that
$$
\frac{d}{ds}  \left\{ \det \left(\frac{\partial \tilde{\Psi}}{\partial h} \bigl(s,h_x\bigr)\right) \right\}_{s=0} = \sum_{i=1}^n \det \left( \left[ \frac{\partial^2 \tilde{\Psi}}{\partial h \partial s} \bigl(0,h_x\bigr) \right]_i \right),
$$ 
where $ \left[ \frac{\partial \tilde{\Psi}}{\partial h \partial s} \bigl(0,h_x\bigr) \right]_i$ is the Jacobian matrix  
$$
\left[ \frac{\partial \tilde{\Psi}}{\partial h_1} \bigl(0,h_x\bigr), \cdots,  \frac{\partial \tilde{\Psi}}{\partial h_n} \bigl(0,h_x\bigr)\right]
$$
whose the $i$-th column is replaced by 
$$
\frac{ \partial^2 \tilde{\Psi}}{\partial h_i \partial s} \bigl(0,h_x\bigr) = \frac{\partial \dot{c}_{h} }{\partial \bar{h}} (0).
$$
More generally, the $k$-th derivative
\begin{eqnarray}\label{18dec2}
\frac{d^k}{ds^k}  \left\{ \det \left(\frac{\partial \tilde{\Psi}}{\partial h} \bigl(s,h_x\bigr)\right) \right\}_{s=0}
\end{eqnarray}
can be expressed as a sum of determinants of matrices of the form
\begin{eqnarray}\label{18dec3}
\left[ \frac{\partial^{1+\beta_1} \tilde{\Psi}}{\partial h \partial s^{\beta_1}} (0), \cdots,  \frac{\partial^{1+\beta_n} \tilde{\Psi}}{\partial h \partial s^{\beta_n}} (0)\right]
\end{eqnarray}
with $\beta_1, \ldots , \beta_n$ some integers verifying
$$
\beta_1, \ldots, \beta_n \geq  0 \quad \mbox{and} \quad \beta_1 + \cdots + \beta_n  \leq k.
$$
By Lemma \ref{LEMl}, we know that 
$$
 \frac{\partial^{2} \tilde{\Psi}}{\partial h \partial s} (0) = \frac{\partial \dot{c}}{\partial h}(0) \in V_1 \quad \mbox{and} \quad  \frac{\partial^{1+l} \tilde{\Psi}}{\partial h \partial s^{l}} (0) = \frac{\partial c^{(l)} }{\partial h}(0) \in  V_1  \oplus \cdots \oplus V_{l-1} \, \forall l \geq 2.
$$
Then in order to be non-vanishing, the sum which gives the $k$-th derivative (\ref{18dec2}) has to contain a term of the form (\ref{18dec3}) with a set $\{\beta_1,\ldots, \beta_n\}$ consisting of at least $m_1$ elements $\geq 1$ and for every $k=2, \ldots, s$, $m_k$ elements $\geq k+1$. This gives the result.
\section{Comments and open problems}  
 
 \subsection{On ideal Carnot groups}\label{SECfat}
Recall that a totally nonholonomic distribution $\Delta$ on a smooth manifold $M$ is called fat if, for every $x\in M$ and every section $X$ of $\Delta$
with $X(x) \neq 0$, there holds
\begin{eqnarray}\label{propfat}
T_xM = \Delta(x)+ \bigl[X,\Delta\bigr](x),
\end{eqnarray}
where 
$$
\bigl[X,\Delta \bigr] (x) := \Bigl\{ [X,Z](x) \, \vert \, Z \mbox{ section of } \Delta \Bigr\}.
$$
Any fat distribution does not admit non-trivial singular curve. As a consequence any fat Carnot group is ideal. Fat Carnot groups are Carnot group of step $s=2$ such that 
$$
\bigl[ V_1, v\bigr] = V_2 \qquad \forall v \in V_1.
$$
This is the case of Heisenberg groups. We refer the reader to \cite{montgomery02,riffordbook} for further details on fat distributions.

Let $\G$ be a Carnot group whose first layer is equipped with a left-invariant metric. Using the same notations as in Section \ref{SECEA} and proceeding as in the proof of Proposition \ref{SREA}, singular curves can be characterized as follows (see \cite{gk95}):

\begin{proposition}
Let $\gamma \in  W^{1,2} \left( [0,T];\G\right)$ be an horizontal curve with $\gamma(0)=0$ associated with a control $u^{\gamma} \in L^2\left( [0,T];\R^m\right)$ such that
\begin{eqnarray*}
\dot{\gamma} (t) = \sum_{i=1}^m u_i^{\gamma}(t) X^1_i\bigl( \gamma(t)\bigr) \qquad \mbox{a.e. } t \in [0,T].
\end{eqnarray*}
Then $\gamma$ is singular if and only if there is an non-vanishing absolutely continuous function
$$
h = \bigl(h^{1}, \ldots, h^s\bigr) \, : \, [0,T] \longrightarrow \R^{m_1} \times  \R^{m_s}
$$
satisfying 
\begin{eqnarray*}
\left\{
\begin{array}{rcl}
\dot{h}^{l}(t) & = & \sum_{k=1}^{m(l+1)}  h_k^{l+1}(t)\, A_k^{1l} \, u^{\gamma}(t) \quad \forall l\in 1, \ldots, s-1,\\
\dot{h}^s(t) & = & 0,
\end{array}
\right.
\qquad \mbox{a.e. } t \in [0,T],
\end{eqnarray*}
such that
\begin{eqnarray*}
h^1 (t) =  0 \qquad \forall t \in [0,T].
\end{eqnarray*}
\end{proposition}

As a consequence, for every $\bar{h} =\left( \bar{h}^1,\ldots, \bar{h}^s\right) \in \R^{n} \setminus \{0\}$ with $\bar{h}^1=0$ such that the linear mapping $\mathcal{S}_{\bar{h}} : \R^m \rightarrow \R^{n-m(s)}$ defined by
$$
\mathcal{S}_{\bar{h}} (v) = \left( \left[ \sum_{k=1}^{m(2)}  \bar{h}_k^{2}\, A_k^{11} \right]\, v, \cdots,  \left[ \sum_{k=1}^{m(s)}  \bar{h}_k^{s}\, A_k^{1(s-1)} \right]\, v \right)
$$
is not injective, there is an horizontal curve $\gamma :[0,1] \rightarrow \G$ associated with a constant control $u^{\gamma} \equiv \bar{u}$ (with $\bar{u} \in \mbox{Ker} (\mathcal{S}_{\bar{h}})$) which is singular. By Proposition \ref{SREA} we check easily that such a curve is the projection of a normal extremal. Since short projections of normal extremal are minimizing (see \cite{riffordbook}), this shows that any Carnot group admitting $\bar{h}  \in \R^{n} \setminus \{0\}$ with $\bar{h}^1=0$  such that $\mathcal{S}_{\bar{h}} $ is not injective is not ideal. For example, this is the case of Carnot groups of step $2$ which are not fat or Carnot groups of step $s\geq 3$ with growth vector $(m_1,\ldots, m_s)$ (remember (\ref{growth})) with $m_r<m_1$ for some $r \in \{2, \ldots, s-1\}$.\\

\subsection{Approximation of SR structures on Carnot groups}
 
The Heisenberg group $\H_1$ equipped with its canonical sub-Riemannian metric is the sub-Riemannian structure $(\R^3,\Delta,g)$ where $\Delta$ is the totally nonholonomic rank $2$ distribution spanned by the vector fields
$$
X  = \partial_{x} - \frac{y}{2} \partial_{z} \quad \mbox{ and } \quad Y = \partial_{y} + \frac{x}{2} \partial_{z},
$$
and $g$ is the metric making $\{X,Y\}$ an orthonormal family of vector fields. A Haar measure is given by the Lebesgue measure $\mathcal{L}^3$. As in \cite{juillet09}, let us introduce the one-parameter family of Riemannian metrics $g_{\epsilon}$ on $\R^3$ which are left-invariant by the Lie group structure and such that the family 
$$
\left\{X, Y, \epsilon \frac{\partial}{\partial z} \right\}
$$ 
is orthonormal. The Hamiltonian $H: \R^3 \times (\R^3)^* \rightarrow \R$ associated with this Riemannian metric is given by 
$$
H\bigl((x,y,z),(p_x,p_y,p_z)\bigr) = \frac{1}{2} \left( p_x - yp_z/2\right)^2 + \frac{1}{2} \left( p_y + xp_z/2\right)^2 + \frac{1}{2} \epsilon^2 p_z^2.
$$
The Hamiltonian system is given by 
\begin{eqnarray}\label{16oct10}
\left\{
\begin{array}{rcl}
\dot{x} & = & p_x - y p_z/2 \\
\dot{y} & = & p_y + x p_z/2 \\
\dot{z} & = & - \left(p_x - yp_z/2\right)y/2  + \left(p_y + xp_z/2\right) x/2 + \epsilon^2 p_z
\end{array}
\right.
\end{eqnarray}
and
\begin{eqnarray}\label{16oct11}
\left\{
\begin{array}{rcl}
\dot{p}_x & = & - \left( p_y + xp_z/2 \right) p_z/2 \\
\dot{p}_y & = & \left( p_x - yp_z/2 \right) p_z/2\\
\dot{p}_z & = & 0.
\end{array}
\right.
\end{eqnarray}
The solution of (\ref{16oct10})-(\ref{16oct11}) starting at $(0,\bar{p}) \in  \R^3 \times (\R^3)^*$ is given by 
\begin{eqnarray}
\left\{ 
\begin{array}{rcl}
x(t) & = & \frac{\bar{p}_y}{\bar{p}_z} \left( \cos \bigl(\bar{p}_zt\bigr)-1\right) + \frac{\bar{p}_x}{\bar{p}_z} \sin \bigl(\bar{p}_zt\bigr) \\
y(t) & = &  - \frac{\bar{p}_x}{\bar{p}_z} \left( \cos \bigl(\bar{p}_zt\bigr)-1\right) + \frac{\bar{p}_y}{\bar{p}_z} \sin \bigl(\bar{p}_zt\bigr) \\
z(t) & = & \epsilon^2 \bar{p}_z t + \frac{\bar{p}_x^2+\bar{p}_y^2}{2\bar{p}_z} \left( t - \frac{\sin(\bar{p}_zt)} {\bar{p}_z} \right)
\end{array}
\right.
\end{eqnarray}
Note that the solutions of (\ref{16oct10})-(\ref{16oct11}) are invariant by rotation. For every $\theta \in \R$, denote by $R_{\theta}$ the rotation of angle $\theta$ with vertical axis. Then we have (here $\exp_0 : T^*\R^3 = \R^3 \times (\R^3)^* \rightarrow \R^3$ denotes the sub-Riemannian exponential mapping, see \cite{riffordbook})
$$
\exp_0 \left( R_{\theta}(\bar{p})\right) = R_{\theta} \left( \exp_0 (\bar{p})\right) \qquad \forall \bar{p} \in \bigl( \R^3\bigr)^*.
$$
Let us now work in cylindrical coordinates, we represent a vector $\bar{p}$ as a triple $(\theta,\rho,p_z)$ and its image by $\exp_0$ as a triple $\left(e_{\theta},e_{\rho}, e_z\right)$ in such a way that in this new set of coordinates we have (with $p_z\neq 0$)
$$
\frac{\partial e_{\theta}}{\partial \theta} =1, \quad \frac{\partial e_{\rho}}{\partial \theta} =0, \quad \frac{\partial e_{z}}{\partial \theta} =0
$$
and
$$
\left\{
\begin{array}{rcl}
e_{\rho} \left( \theta, \rho, p_z\right) & = & \rho  \left| \frac{\sin (p_z/2)}{p_z/2}\right| \\
e_{z} \left( \theta, \rho, p_z\right) & = & \epsilon^2 p_z + \frac{\rho^2}{2p_z} \left( 1- \frac{\sin(p_z)}{p_z}\right).
\end{array}
\right.
$$
Denoting by $\tilde{\exp}_0$ the exponential mapping in this new set of coordinates, we have for every triple  $(\theta,\rho,p_z)$, 
\begin{eqnarray*}
\mbox{Jac}_{\theta,\rho,p_z} \tilde{\exp}_0 & = & \det \left( \begin{matrix}
\frac{\partial e_{\rho}}{\partial \rho} & \frac{\partial e_{\rho}}{\partial p_z} \\
 \frac{\partial e_{z}}{\partial \rho} & \frac{\partial e_{z}}{\partial p_z} 
 \end{matrix}
 \right) \\
 & = & \epsilon^2 \frac{\sin \bigl(p_z/2\bigr)}{p_z/2} + \frac{2\rho^2}{p_z^3} \Bigl( \sin \bigl( p_z/2) - \bigl( p_z/2\bigr) \cos \bigl(p_z/2\bigr) \Bigr).
\end{eqnarray*}
Therefore, for every measurable set $A \subset \R^3$ with $0< \mathcal{L}^3(A)<\infty$, we have (we denote by $\tilde{A}$ the set $A$ in our set of coordinates)
$$
\mathcal{L}^3 (A) = \int_{\tilde{A}_1} e_{\rho} \, de_{\theta} \, de_{\rho} \, de_z  =  \int_{\tilde{D}_1} \rho  \left| \frac{\sin (p_z/2)}{p_z/2}\right| \left| \mbox{Jac}_{\theta,\rho,p_z} \tilde{\exp}_0 \right|   \, d\theta \, d\rho \, dp_z,
$$
and for every $s\in [0,1]$,
\begin{eqnarray*}
\mathcal{L}^3(A_s) = \int_{\tilde{A}_s} e_{\rho} \, de_{\theta} \, de_{\rho} \, de_z & =  & \int_{\tilde{D}_s} \rho  \left| \frac{\sin (p_z/2)}{p_z/2}\right| \left| \mbox{Jac}_{\theta,\rho,p_z} \tilde{\exp}_0 \right|   \, d\theta \, d\rho \, dp_z \\
& =  & \int_{\tilde{D}_1} s^3 \rho  \left| \frac{\sin (sp_z/2)}{sp_z/2}\right| \left| \mbox{Jac}_{\theta,s\rho,sp_z} \tilde{\exp}_0 \right|   \, d\theta \, d\rho \, dp_z.
\end{eqnarray*} 
Define the function $h,k: (0,\pi) \rightarrow \R$ by 
$$
h(\lambda) := \frac{\sin (\lambda)}{\lambda}, \quad k(\lambda) :=   \sin ( \lambda) - \lambda \cos (\lambda) \qquad \forall \lambda \in (0,\pi).
$$
We check easily that the functions $\lambda \mapsto h(\lambda)$ and $\lambda \mapsto h(\lambda)k(\lambda)/\lambda^3$ are positive and decreasing on $(0,\pi)$. Then we have for any $s\in (0,1),p_z\neq 0$,
\begin{multline*}
   s^3 \rho h\bigl( sp_z/2\bigr) \left[ \epsilon^2 h\bigl( sp_z/2\bigr) + \frac{2\rho^2}{sp_z^3} k \bigl( sp_z/2\bigr) \right] \\
=   s^3 \rho  \epsilon^2 h\bigl( sp_z/2\bigr)^2 + 2 \rho^3 s^5 \left( \frac{ h\bigl( sp_z/2\bigr) h\bigl( sp_z/2\bigr) }{s^3p_z^3}\right) \\
\geq s^3 \rho  \epsilon^2 h\bigl( p_z/2\bigr)^2 + 2 \rho^3 s^5 \left( \frac{ h\bigl( p_z/2\bigr) h\bigl( p_z/2\bigr) }{p_z^3}\right)\\
\geq s^5  \rho h\bigl( p_z/2\bigr) \left[ \epsilon^2 h\bigl( p_z/2\bigr) + \frac{2\rho^2}{p_z^3} k \bigl( p_z/2\bigr) \right].
   \end{multline*}
   All in all, we get
 $$
 \mathcal{L}^3(A_s) \geq s^5 \mathcal{L}^3(A) \qquad \forall s\in [0,1].
 $$  
We leave the reader to check that the above discussion implies that the Heisenberg group $H_1$ equipped with the left-invariant Riemannian metric $g_{\epsilon}$ (with $\epsilon>0$) satisfies $\MCP(0,5)$. It also implies (taking $\epsilon=0$), as checked by Juillet \cite{juillet09}, that the Heisenberg group equipped with its canonical sub-Riemannian metric satisfies $\MCP(0,5)$. This means that the measured metric space $(H_1,d_{SR},\mathcal{L}^3)$ can be approximated (in Gromov-Hausdorff topology, see \cite{villanibook}) by a sequence of Riemannian measured metric spaces with the same curvature exponent.  We do not know if such property holds for more general Carnot groups with finite curvature exponent. Let $\G$ be a Carnot group whose first layer is equipped with a left-invariant metric $g$, assume that  it is a geodesic space with negligeable cut loci and that it satisfies $\MCP(0,N)$. Does there exists a sequence of left-invariant Riemannian metrics $\{g_{\epsilon}\}_{\epsilon>0}$ converging to $g$ together with a sequence $\{N_{\epsilon}\}_{\epsilon>0}$ converging to $N$ as $\epsilon \downarrow 0$ such that each Riemannian space $(\G,g_{\epsilon})$ satisfies $\MCP(0,N_{\epsilon})$ ? Can we expect such a result for Carnot groups which admit left-invariant Riemannian metrics with positive Ricci curvatures on the orthogonal complement of the first layer (as for $H_1$, see \cite{juillet09,milnor76}) ?

\subsection{Step $2$ Carnot groups}

Recall that a totally nonholonomic distribution $\Delta$ on a smooth manifold $M$ is called  two-generating or of step $2$ in $M$ if 
$$
T_xM = \Delta(x) + [\Delta,\Delta](x) \qquad \forall x \in M,
$$
where $[\Delta,\Delta]$ is defined as
$$
[\Delta,\Delta](x) := \Bigl\{ [X,Y](x) \, \vert \, X, Y \mbox{ sections of } \Delta \Bigr\}.
$$
In \cite{al09} (see also \cite{riffordbook}), Agrachev and Lee proved that any complete sub-Riemannian structure with a distribution of step $2$ is Lipschitz, that is its sub-Riemannian distance is locally Lipschitz outside the diagonal in $M\times M$. The lipschitzness of the sub-Riemannian distance outside the diagonal allows to recover assertions (ii)-(iv) of Proposition \ref{PROPideal} for some closed set in $M$ (which may  be bigger that the one defined in (\ref{cutSR}) (see \cite{agrachev09,riffordbook})).

\begin{proposition}\label{PROPrio}
Let $(\Delta,g)$ be a Lipschitz sub-Riemannian structure on $M$. Then for every $x\in M$, there is a closed set $\mathcal{C}(x)\subset M$ of Lebesgue measure zero such that the pointed distance $d_{SR}(x,\cdot)$ is smooth on the open set $M\setminus \mathcal{C}(x)$ and for every $y\in M\setminus \mathcal{C} (x)$ there is only one minimizing geodesic between $x$ and $y$ and it is not singular.
\end{proposition}

Therefore, any Carnot group $\G$ of step $2$ is Lipschitz and geodesic with negligeable cut loci. By invariance by dilations and translations (as seen in the proof of Theorem \ref{THM1}), a $2$ step Carnot group $\G$ satisfies $\MCP(0,N)$ for some $N>1$ if and only if there is $N>0$ such that for every measurable set $A  \subset B_{SR}(0,1) \setminus B_{SR}(0,1/2)$, 
\begin{eqnarray}\label{noel}
\vol_{\G} (A_s) \geq s^N \, \vol_{\G} (A) \qquad \forall s \in [1/2,1].
\end{eqnarray}
The above proposition allows to show that (\ref{noel}) holds far from $\mathcal{C}(x)$ and indeed at least far from conjugate points. The validity of some MCP property depends on the behavior of the sub-Riemannian exponential mapping near conjugate points.

\subsection{Other notions of synthetic Ricci curvature bounds}

In the present paper, we have restricted our attention to the Ohta Measure Contraction Property. Many other notions of synthetic Ricci curvature bounds do exist, we refer the reader to \cite{juillet09} and \cite{villanibook} for further details. In particular, in \cite{juillet09,juilletkyoto}, Juillet checked that the canonical Sub-Riemannian on the Heisenberg group does not satisfy curvature dimension conditions in the sense of Lott-Villani \cite{lv07} and Sturm \cite{sturm06a,sturm06b}. We may expect that more general Carnot groups do not satisfy those conditions.


\addcontentsline{toc}{section}{References}
\begin{thebibliography}{99}

\bibitem{agrachev09} 
A.~Agrachev.
\newblock Any sub-Riemannian metric has points of smoothness. 
\newblock {\em Dokl. Akad. Nauk}, 424(3): 295--298, 2009
\newblock Translation in {\em Dokl. Math.}, 79(1):45--47, 2009.

\bibitem{abb12}
A.~Agrachev, D.~Barilari and U.~Boscain.
\newblock Introduction to Riemannian and sub-Riemannian geometry.
\newblock To appear.

\bibitem{al09}
A.~Agrachev and P.~Lee.
\newblock Optimal transportation under nonholonomic constraints.
\newblock  {\em Trans. Amer. Math. Soc.}, 361(11):6019--6047, 2009.

\bibitem{al12}
A.~Agrachev and P.~Lee.
\newblock Generalized Ricci curvature bounds for three dimensional contact subriemannian manifolds.
\newblock  Preprint, 2009.

\bibitem{arnold66}
V.~Arnold.
\newblock Sur la g\'eom\'etrie diff\'erentielle des groupes de Lie de dimension infinie et ses applications \`a l'hydrodynamique des fluides parfaits.
\newblock {\em Ann. Inst. Fourier (Grenoble)}, 16(1):319--361, 1966.

\bibitem{bellaiche96}
A.~Bella\"iche.
\newblock The tangent space in sub-Riemannian geometry.
\newblock In {\em Sub-Riemannian Geometry}, Birkh\"auser, 1--78, 1996.

\bibitem{cr08}
P.~Cannarsa and L.~Rifford.
\newblock Semiconcavity results for optimal control problems admitting
no singular minimizing controls.
\newblock {\em Ann. Inst. H. Poincar\'e Non Lin\'eaire}, 25(4):773--802, 2008.

\bibitem{cs04}
P.~Cannarsa and C.~Sinestrari.
\newblock {\em Semiconcave functions, Hamilton-Jacobi equations, and optimal control}.
\newblock Progress in Nonlinear Differential Equations and their Applications, 58. Birkh\"auser Boston Inc., Boston, MA, 2004.

\bibitem{cr10}
M.~Castelpietra and L.~Rifford.
\newblock Regularity properties of the distance function to conjugate and cut loci for viscosity solutions of Hamilton-Jacobi equations and applications in Riemannian geometry.
\newblock {\em ESAIM Control Optim. Calc. Var.}, 16(3):695--718, 2010.

\bibitem{fr10}
A.~Figalli and L.~Rifford.
\newblock Mass Transportation on sub-Riemannian Manifolds.
\newblock {\em Geom. Funct. Anal.}, 20(1):124--159, 2010.

\bibitem{ghl04}
S.~Gallot, D.~Hulin and J.~Lafontaine.
\newblock {\em Riemannian geometry}.
\newblock Third edition. Universitext. Springer-Verlag, Berlin, 2004.

\bibitem{gk95}
C.~Gol\'e and R.~Karidi.
\newblock A note on Carnot geodesics in nilpotent Lie groups.
\newblock  {\em J. Dynam. Control Systems}, 1(4):535--549, 1995.

\bibitem{gromov99}
M.~Gromov.
\newblock {\em Metric structures for Riemannian and non-Riemannian spaces.}
\newblock Progress in Mathematics, vol. 152. Birkh\"auser Boston Inc. Boston, MA, 1999.

\bibitem{it01}
J.~Itoh and M.~Tanaka.
\newblock The Lipschitz continuity of the distance function to the cut locus.
\newblock {\em Trans. Amer. Math. Soc.}, 353(1):21--40, 2001.

\bibitem{juillet09}
N.~Juillet.
\newblock Geometric inequalities and generalized Ricci bounds in the Heisenberg group.
\newblock {\em Int. Math. Res. Not. IMRN}, 13:2347--2373, 2009. 

\bibitem{juilletkyoto}
N.~Juillet.
\newblock On a method to disprove generalized Brunn-Minkowski inequalities.
\newblock In {\em Probabilistic approach to geometry}, 189--198, Adv. Stud. Pure. Math., 57, Math. Soc. Japan, Tokyo, 2010. 

\bibitem{ledonne10}
E.~Le Donne.
\newblock Lecture notes on sub-Riemannian geometry.
\newblock Preprint, 2010.
 
\bibitem{ln05}
Y.~Li and L.~Nirenberg.
\newblock The distance function to the boundary, Finsler geometry, and the singular set of
viscosity solutions of some Hamilton-Jacobi equations.
\newblock  {\em Comm. Pure Appl. Math.}, 58(1):85--146, 2005.

\bibitem{lv07}
J.~Lott and C.~Villani.
\newblock Weak curvature conditions and functional inequalities.
\newblock {\em J. Funct. Anal.}, 245(1):311--333, 2007. 

\bibitem{milnor76}
J.~Milnor.
\newblock Curvatures of left-invariant metrics on Lie groups.
\newblock {\em Advances in Math.}, 21(3):293--329, 1976. 

\bibitem{mitchell85}
J.~Mitchell.
\newblock On Carnot-Carath\'eodory spaces.
\newblock {\em J. Differential Geom.}, 21(9):35--45, 1985. 

\bibitem{montgomery02}
R.~Montgomery.
\newblock {\em A tour of sub-Riemannian geometries,
their geodesics and applications}.
\newblock {\em Mathematical Surveys and Monographs}, Vol.\ 91.
\newblock American Mathematical Society, Providence, RI, 2002.

\bibitem{ohta07}
S.~Ohta.
\newblock On the measure contraction property of metric measure spaces.
\newblock {\em Comment. Math. Helv.}, 82(4):805--828, 2007.

\bibitem{riffordbook}
L.~Rifford.
\newblock {\em Sub-Riemannian Geometry and Optimal Transport}.
\newblock Preprint, 2012.

\bibitem{sturm06a}
K.~-T.~Sturm.
\newblock On the geometry of metric measure spaces. I. 
\newblock {\em Acta Math.}, 196(1):65--131, 2006.

\bibitem{sturm06b}
K.~-T.~Sturm.
\newblock On the geometry of metric measure spaces. II. 
\newblock {\em Acta Math.}, 196(1):133--177, 2006.

\bibitem{villanibook} 
C.~Villani.
\newblock {\em Optimal transport, Old and New}.
\newblock Grundlehren der Mathematischen Wissenschaften, 338. Springer-Verlag, Berlin, 2009.


\end{thebibliography}
\end{document}